\documentclass[12pt]{amsart}
\usepackage{amsmath,amsthm,amsfonts,latexsym,amssymb,amscd,color}
\usepackage{chemarr}
\usepackage{graphicx}
\usepackage{mdframed}
\newtheorem{thm}{Theorem}[section]
\newtheorem{cor}[thm]{Corollary}
\newtheorem{lm}[thm]{Lemma}
\newtheorem{clm}[thm]{Claim}
\newtheorem{subclm}[thm]{Subclaim}
\theoremstyle{definition}
\newtheorem{df}[thm]{Definition}
\newtheorem{exmp}[thm]{Example}

\newtheorem{prb}[thm]{Problem}
\newtheorem{rem}[thm]{Remark}
\numberwithin{figure}{section}
\numberwithin{equation}{section}
\newcommand{\lb}{\langle} 
\newcommand{\rb}{\rangle}
\newcommand{\wec}[1]{{\mathbf{#1}}}  
\newcommand{\m}[1]{{\mathbf{\uppercase{#1}}}}
\newcommand{\vr}[1]{{\mathcal{\uppercase {#1}}}}

\newcommand{\BZ}{\mathbb{Z}} 
\newcommand{\Ho}{\mathsf{H}}

\newcommand{\Pol}{\mathop{}\mathopen{}\mathrm{Pol}}

\newcommand{\Con}{\mathop{}\mathopen{}\mathrm{Con}}
\newcommand{\Conb}{\mathop{}\mathopen{}\mathbf{Con}} 

\newcommand{\utyp}{\mathbf{1}}
\newcommand{\atyp}{\mathbf{2}}
\newcommand{\btyp}{\mathbf{3}}
\newcommand{\ltyp}{\mathbf{4}}
\newcommand{\styp}{\mathbf{5}}
\newcommand{\NC}{\mathbf{C}} 
\newcommand{\Tr}{\mathbf{Tr}} 
\newcommand{\cproof}{\noindent{\it Proof of claim.}\ } 
\newcommand{\cqed}{\hfill\rule{1.3mm}{3mm}}

\newcommand{\subcproof}{\noindent{\it Proof of subclaim.}\ } 

\advance\parskip by 0pt plus 2pt 
\begin{document}
\title[Growth Rates of Solvable Algebras]{Growth Rates of Algebras, III:\\
\vspace{2mm} {\normalsize\rm Finite Solvable Algebras}}

\author{Keith A. Kearnes}
\address[Keith Kearnes]{Department of Mathematics\\
University of Colorado\\
Boulder, CO 80309-0395\\
USA}
\email{Keith.Kearnes@Colorado.EDU}
\author{Emil W. Kiss}
\address[Emil W. Kiss]{
Lor\'{a}nd E{\"o}tv{\"o}s University\\
Department of Algebra and Number Theory\\
1117 Budapest, P\'{a}zm\'{a}ny P\'{e}ter s\'{e}t\'{a}ny 1/c\\
Hungary}
\email{ewkiss@cs.elte.hu}
\author{\'Agnes Szendrei}
\address[\'Agnes Szendrei]{Department of Mathematics\\
University of Colorado\\
Boulder, CO 80309-0395\\
USA}
\email{Agnes.Szendrei@Colorado.EDU}
\thanks{This material is based upon work supported by
the Hungarian National Foundation for Scientific Research (OTKA)
grant no.\ K83219, and K104251.
}
\dedicatory{Dedicated to the memory of Ervin Fried}

\begin{abstract}
  We investigate how the behavior of the function 
  $d_{\m a}(n)$, which gives the size of a least size generating
  set for~$\m a^n$, influences the structure of a finite
  solvable algebra~$\m A$.
\end{abstract}

\maketitle
\section{Introduction}\label{introduction}
The \emph{growth rate} (or the $d$-function) of a finite
algebra~$\m A$ is $d_{\m A}(n)$ = the least size of a
generating set for $\m A^n$. For a solvable group, this rate
is always linear in~$n$. On the other hand, unary algebras
(which are also solvable) have exponential growth rate. In
this paper we investigate the relationship between the
growth rate of $\m A$ and its structural properties in the
case, when $\m A$ is finite and solvable.

It turns out that stronger abelianness properties yield a
closer relationship between various growth-restricting
conditions. For example, if the variety generated by $\m A$
is abelian, or if $\m A$ is a subdirect product of simple
abelian algebras, then the growth rate is non-exponential if
and only if $\m A$ has a Maltsev term, in which case the
growth rate is linear. This result does not hold for
nilpotent (hence for solvable) algebras.

To define the abelianness properties investigated in this
paper the commutator from Chapter~3 of \cite{hobby-mckenzie}
is used. Some fluency in tame congruence theory is required
to understand the proofs. The properties are the
following:
\begin{enumerate}
\item[(1)] $\m a$ is solvable, 
\item[(2)] $\m a$ is (left) nilpotent (see \cite{sym}),
\item[(3)] $\m a$ is abelian, 
\item[(4)] $\m a$ is a subdirect product of simple abelian algebras, or
\item[(5)] $\m a$ generates an abelian variety.
\end{enumerate}

It is not hard to show from the definitions that
these properties are related by 
the implications
(4)$\Rightarrow$(3)$\Rightarrow$(2)$\Rightarrow$(1)
and
(5)$\Rightarrow$(3). 
It is also not too hard to show that 
no other implications hold except those that are 
formal consequences of these.
(To verify this latter statement, one must find examples showing that 
(1)~$\not\Rightarrow$~(2), 
(2)~$\not\Rightarrow$~(3), 
(3)~$\not\Rightarrow$~(4), 
(3)~$\not\Rightarrow$~(5), 
(4)~$\not\Rightarrow$~(5), and
(5)~$\not\Rightarrow$~(4), and for many
of these there exists an example that is a group.
In the two situations where there is no group example,
(3)~$\not\Rightarrow$~(5) and 
(4)~$\not\Rightarrow$~(5), 
the algebra
from Example~8.7 of \cite{strnil} shows that a simple abelian
algebra can generate a nonabelian variety.)

\goodbreak

Now consider the following six growth-restricting 
conditions:
\begin{enumerate}
\item[(i)]
$\m a$ has a Maltsev polynomial.
\item[(ii)]
$\m a$ has a pointed cube polynomial
(see Definition~\ref{pointed_cube}).
\item[(iii)] $\m A$ is a spread of its type~$\atyp$ minimal sets
  (see Definition~\ref{spread}).
\item[(iv)]
$d_{\m a}(n)\in O(n)$.
\item[(v)]
$d_{\m a}(n)\notin 2^{\Omega(n)}$.
\item[(vi)]
No finite power $\m a^n$ has a nontrivial
strongly abelian homomorphic image.
\end{enumerate}
For arbitrary finite algebras, 
these conditions are related by the implications
(i)$\Rightarrow$(iv),
(i)$\Rightarrow$(ii)$\Rightarrow$(v), 
and
(iii)$\Rightarrow$(iv)$\Rightarrow$(v)$\Rightarrow$(vi)
and no
implications hold that are not consequences of these,
as we show in Theorem~\ref{basic}.

In this paper we show that as we
assume stronger and stronger hypotheses from
the list (1)--(5),
the conditions (i)--(vi) gradually
become equivalent to one another.
The following theorem summarizes our main results.

\begin{thm}\label{main}
  Let~$\m A$ be a finite algebra.
  \begin{enumerate}
  \item If~$\m A$ is solvable, then (i)$\Rightarrow$(iii) by
   Corollary~\ref{kkv_cor} and (ii)$\Rightarrow$(iv) by
    Theorem~\ref{1-solv}.
  \item If~$\m A$ is left nilpotent, then
    (ii)$\Rightarrow$(i) by Theorem~\ref{nilpotent} and
    (vi)$\Rightarrow$(iii) by
    Theorem~\ref{spreadnilp}. Therefore 
    (i) and (ii) are equivalent, (iii), (iv), (v) and
    (vi) are equivalent, and the first group of equivalent
    conditions implies the second group. 
    The same implications (and no others, 
    cf.\ Example~\ref{abelian_spread_not_maltsev}) hold
    under the stronger assumption that $\m a$ is abelian.
  \item If $\m A$ is a subdirect product of simple abelian
    algebras, then (iv)$\Rightarrow$(i) by
    Theorem~\ref{simplefact}. Therefore all six conditions
    are equivalent for $\m a$.
  \item If $\m A$ generates an abelian variety, then
    (iii)$\Rightarrow$(i) by
    Theorem~\ref{spreadaff}. Therefore all six conditions
    are equivalent for $\m a$.
  \end{enumerate}
\end{thm}

These results are expressed diagrammatically in Section~\ref{summary_sec}.

Our theorem completely determines the relationships between the 
growth-restrict\-ing conditions (i)--(vi) for 
nilpotent and abelian
algebras, but some questions remain open for solvable algebras.
We do know that 
(i)$\Rightarrow$(ii)$\Rightarrow$(iv)$\Rightarrow$(v)$\Rightarrow$(vi)
and (i)$\Rightarrow$(iii)$\Rightarrow$(iv) 
for solvable algebras, and we also know that 
(iii)$\not\Rightarrow$(ii)$\not\Rightarrow$(i).
We know nothing else about the implications between
these properties for finite solvable
algebras that are not nilpotent.

The original purpose of our investigation was to show that
the $d$-function of a finite solvable algebra grows at a
linear or exponential rate. We were not able to prove or
refute this statement. We considered the other conditions in
order to better understand what might force the $d$-function
of a solvable algebra into $O(n)$ or $2^{\Omega(n)}$.  If,
for example, one could now show that (vi)$\Rightarrow$(iii)
for finite solvable algebras, then one would have that the
$d$-functions of such algebras grow linearly or
exponentially.

\section{Preliminaries}\label{quoted}

\subsection{Notation}
We use Big Oh notation.  If $f$ and $g$ are real-valued
functions defined on some subset of the real numbers, then
$f\in O(g)$ means that there are positive constants $M$ and
$N$ such that $|f(x)|\leq M|g(x)|$ for all $x>N$.  We write
$f\in \Omega(g)$ to mean that there are positive constants
$M$ and $N$ such that $|f(x)|\geq M|g(x)|$ for all $x>N$.
Finally, $f\in \Theta(g)$ means that both $f\in O(g)$ and
$f\in \Omega(g)$ hold.

\subsection{Easy Estimates}

\begin{thm}[Theorem 2.2.1 of \cite{kksz-A}]\label{basic_estimates} 
Let $\m a$ be an algebra.
\begin{enumerate}
\item[(1)] $d_{\m a^k}(n) = d_{\m a}(kn)$.
\item[(2)] If\/ $\m b$ is a homomorphic image of 
$\m a$, then $d_{\m b}(n) \leq d_{\m a}(n)$.
\item[(3)] If\/ $\m b$ is an expansion of $\m a$ (with
  operations; equivalently, if $\m a$ is a reduct of\/~$\m
  b$), then $d_{\m b}(n) \leq d_{\m a}(n)$.
\item[(4)] {\rm(From \cite{quick-ruskuc})} If\/ $\m b$ is
  the expansion of $\m a$ obtained by adjoining all
  constants, then
\[
d_{\m a}(n) - d_{\m a}(1)\leq d_{\m b}(n) \leq d_{\m a}(n). 
\]
\end{enumerate}

\kern-10pt

\qed
\end{thm}

\begin{thm}[From Theorem 2.2.2 of \cite{kksz-A}]\label{first_bounds} 
If $\m a$ is a finite algebra of
more than one element and $n>0$, then 
\[
\lceil \log_{|A|}(n) \rceil\leq d_{\m a}(n) \leq |A|^n. 
\]

\kern-10pt

\qed
\end{thm}

Recall that the \emph{free spectrum} of a variety
$\mathcal V$ is the function $f_{\mathcal V}(n):=|F_{\mathcal V}(n)|$
whose value at $n$ is the cardinality of the
$n$-generated free algebra in $\mathcal V$.

\begin{thm}[From Theorem 2.2.4 of \cite{kksz-A}]\label{spec1}
If $\m a$ is a nontrivial finite algebra and $f_{\mathcal V}$
is the free spectrum of the variety $\mathcal V = {\mathcal V}(\m a)$,
then 
\begin{enumerate}
\item[(1)] 
if $f_{\mathcal V}(n)\in O(n^k)$ for some fixed $k\in\mathbb Z^+$, 
then $d_{\m a}(n)\in 2^{\Theta(n)}$;
\item[(2)] 
if $f_{\mathcal V}(n)\in 2^{O(n)}$,
then $d_{\m a}(n)\in \Omega(n)$. \qed
\end{enumerate}
\end{thm}

\begin{cor}[Corollary 2.2.5 of \cite{kksz-A}]\label{abelian_cor}
Let $\m a$ be a nontrivial finite algebra and let\/~$\m b$
be a nontrivial homomorphic image of $\m a^k$ for some $k$.
\begin{enumerate}
\item[(1)] 
If\/ $\m b$ is strongly abelian (or even just strongly rectangular), 
then $d_{\m a}(n)\in 2^{\Theta(n)}$.
\item[(2)] 
If\/ $\m b$ is abelian, then $d_{\m a}(n)\in \Omega(n)$. \qed
\end{enumerate}
\end{cor}

\begin{thm}[Theorem 2.2.6 of \cite{kksz-A}]\label{affine_prep}
If $\m a^2$ is a finitely generated affine algebra, 
then $d_{\m a}(n)\in O(n)$. If, moreover, $\m a$ is finite
and has more than one element, then 
$d_{\m a}(n)\in \Theta(n)$. 
\qed
\end{thm}

\subsection{Basic relationships among conditions (i)--(vi)}

Recall that a polynomial $F(x,y,z)$ of $\m a$ is a \emph{Maltsev polynomial}
if $F(x,y,y)\approx x \approx F(y,y,x)$ holds in $\m a$.
In \cite{kksz-A}, this well known concept is generalized
to the following:

\begin{df}\label{pointed_cube}
A term $F(x_1,\ldots,x_m)$ is
an \emph{$m$-ary, $p$-pointed, $k$-cube term} for $\m a$
if there is a $k\times m$ matrix $M$
consisting of variables and $p$ distinct constant symbols,
with every column of $M$ containing a symbol different 
from $x$, such that $\m a$ satisfies
\begin{equation}\label{cube}
F(M)\approx 
\left(
\begin{matrix}
x\\
\vdots\\
x
\end{matrix}
\right).
\end{equation}
\end{df}

For example, 
\[
{F}\left(\begin{matrix}
x&y&y\\
y&y&x
\end{matrix}
\right) \approx
\left(\begin{matrix}
x\\
x
\end{matrix}
\right),
\]
is a way to express 
that $F$ is a Maltsev term in the form (\ref{cube})
with $m=3$, $p=0$ and $k=2$.

The main results from \cite{kksz-A}
about the restriction on growth
imposed by a pointed cube term
now follow.

\begin{thm}[Theorem 5.2.1 of \cite{kksz-A}]\label{pointed_polynomial}
  Let $\m a$ be an algebra with an $m$-ary, $p$-pointed,
  $k$-cube term, with at least one constant symbol appearing
  in the cube identities (that is, $p\geq 1$). If\/ $\m
  a^{p+k-1}$ is finitely generated, then all finite powers
  of $\m a$ are finitely generated and $d_{\m a}(n)$ is
  bounded above by a polynomial of degree at most
  $\log_w(m)$, where $w = 2k/(2k-1)$.  \qed
\end{thm}

\begin{cor}[Corollary 5.2.4 of \cite{kksz-A}]\label{pointed_polynomial_cor}
If $\m a^k$ is a finitely generated algebra with a $0$-pointed
or $1$-pointed $k$-cube term, then $d_{\m a}(n)\in O(n^{k-1})$. 
\qed
\end{cor}


\begin{df}\label{pointed_cube_pol}
An $m$-ary, $p$-pointed, $k$-cube term for the constant expansion of $\m a$
is called an \emph{$m$-ary, $p$-pointed, $k$-cube polynomial} for $\m a$. 
\end{df}

\begin{rem}
An algebra has a Maltsev polynomial (or a pointed cube polynomial)
if and only if its constant expansion has a Maltsev term
(or pointed cube term).
Passing to the constant expansion does not affect
the growth rate significantly, as noted in 
Theorem~\ref{basic_estimates}~(4).
Therefore, the last two 
results hold for 
polynomials as well as for terms. 
In fact, replacing ``term'' with ``polynomial''
in Theorem~\ref{pointed_polynomial} and
Corollary~\ref{pointed_polynomial_cor}
yields correct statements. 

In the literature, a $0$-pointed cube term is called a
cube-term. Note that a $p$-pointed cube term is not
necessarily a cube-term even if all constants are unary
terms.
\end{rem}

\begin{thm}[From Theorem 5.4.1 of \cite{kksz-A}]\label{avoid}
Let $\m a$ be an algebra with $|A|>1$ whose
$d$-function assumes only finite values.
There
is an algebra $\m b$ such that $d_{\m b}(n) = d_{\m a}(n)$ for all $n$, 
where $\m b$ does not have a pointed cube polynomial.
\qed
\end{thm}

For the other concepts appearing in the conditions (i)--(vi),
we direct the reader to Definition~\ref{spread} for 
``spread'' and to \cite{hobby-mckenzie} for
``minimal set'' and ``strongly abelian''.

We can now establish which implications between 
conditions (i)--(vi) from the Introduction
hold for all finite algebras.

\begin{thm}\label{basic}
For arbitrary finite algebras, the conditions
\begin{enumerate}
\item[(i)]
$\m a$ has a Maltsev polynomial.
\item[(ii)]
$\m a$ has a pointed cube polynomial.
\item[(iii)] $\m A$ is a spread of its type~$\atyp$ minimal sets.
\item[(iv)]
$d_{\m a}(n)\in O(n)$.
\item[(v)]
$d_{\m a}(n)\notin 2^{\Omega(n)}$.
\item[(vi)]
No finite power $\m a^n$ has a nontrivial
strongly abelian homomorphic image.
\end{enumerate}
are related by the implications
\[
(i)\Rightarrow(iv),\quad
(i)\Rightarrow(ii)\Rightarrow(v), \quad
\textrm{and}\quad
(iii)\Rightarrow(iv)\Rightarrow(v)\Rightarrow(vi).
\]
The following non-implications 
can be established with finite counterexamples:
\[
(vi)\not\Rightarrow
(v)\not\Rightarrow
(iv)\not\Rightarrow
(iii)\not\Rightarrow
(ii)\not\Rightarrow
(iv)\not\Rightarrow
(ii)
\quad
\textrm{and}\quad
(i)\not\Rightarrow
(iii).
\]

These facts imply that the implications that hold are 
only those indicated in
the following diagram and their consequences:
\bigskip

\begin{center}
\begin{picture}(200,60)
\setlength{\unitlength}{1mm}

\put(0,0){$(iii)$}
\put(10,0){$\Longrightarrow$}
\put(20,0){$(iv)$}
\put(30,0){$\Longrightarrow$}
\put(40.5,0){$(v)$}
\put(50,0){$\Longrightarrow$}
\put(60,0){$(vi)$.}
\put(22,10){\rotatebox[origin=c]{270}{$\Longrightarrow$}}
\put(21,20){$(i)$}
\put(30,20){$\Longrightarrow$}
\put(40,20){$(ii)$}
\put(42,10){\rotatebox[origin=c]{270}{$\Longrightarrow$}}
\end{picture}
\end{center}

\medskip

\end{thm}

We shall prove this theorem in Section~\ref{summary_sec}.

\section{A new characterization of solvability}\label{solvchar}

In this section we show how elementary translations coming
from idempotent polynomials characterize solvability.

\begin{df}\label{t-digraph}
  Let $\m A$ be an algebra and $p$ an idempotent polynomial
  of~$\m A$ (that is, $p(x,x,\ldots,x)=x$ for every $x\in
  A$). The \emph{translation-digraph} $\Tr(p)$ of $p$ has
  vertex set~$A$, and directed edges of the form
  $(c,c')=\big(p(c,c,\ldots,c),
  p(c,\ldots,c,d,c,\ldots,c)\big)$, where $c,d\in A$. (The
  element $d$ occurs in exactly one of the arguments of $p$,
  but it can be any of the arguments.)
\end{df}

Recall that a \emph{neighborhood} $U$ of an algebra $\m A$
is the range of an idempotent unary polynomial of $\m A$,
that is, $U=e(A)$ for a unary polynomial $e$ satisfying
$e\big(e(x)\big)=e(x)$ for every $x\in A$.

\begin{thm}\label{solvcharthm}
  Let $\m A$ be a finite algebra. Then $\m A$ is solvable if
  and only if for every neighborhood $U$ of~$\m A$, and
  every idempotent polynomial $p$ of the induced algebra~$\m
  A|_U$, the directed graph $\Tr(p)$ is strongly connected.
\end{thm}

The proof of Theorem~\ref{solvcharthm}
is included at the end of this section.
The essential part of the proof is to establish the theorem
for classes of solvable minimal congruences.

\begin{lm}\label{transl}
  Let $\m S$ be a finite, simple, abelian algebra and $p$ an
  idempotent polynomial of\/~$\m S$. Then $\Tr(p)$ is
  strongly connected.
\end{lm}

\begin{proof}
  Since $S$ is connected by traces, it is sufficient to
  prove that for every trace~$N$ and every $a,b\in N$ there
  is a directed path in $\Tr(p)$ connecting $a$
  to~$b$. Consider the multitraces with respect to~$N$, that
  is, all sets of the form $q(N,N,\ldots,N)$, where $q$ is
  any polynomial of~$\m S$. Choose a multitrace $M$ that
  contains~$N$ and is maximal under inclusion. Then
  $p(M,\ldots,M)$ contains~$M$ (and therefore $N$), since
  $p$ is idempotent, and is a multitrace, so
  $p(M,\ldots,M)=M$. Theorem~3.10 of \cite{kkv1} shows that
  the induced algebra $\m S|_M$ is polynomially equivalent
  to the full matrix power $(\m S|_N)^{[k]}$ for some~$k$
  (more precisely, $\m S|_M$ is isomorphic to an algebra on
  $N^k$ that is polynomially equivalent to $(\m
  S|_N)^{[k]}$).

  Consider first the case, when the type of~$\m S$
  is~$\utyp$. Let $a$ correspond to the vector
  $(a_1,\ldots,a_k)$ and $b$ correspond to
  $(b_1,\ldots,b_k)$. We may assume that $a$ and $b$ differ
  only in one coordinate, and, to simplify the notation,
  that this is the first coordinate. That is, $a_i=b_i$ for
  $i\ge 2$. Indeed, if such pairs can be connected by a
  directed path, then by changing only one coordinate at a
  time we can connect $a$ to $b$.

  By the definition of a matrix power, the operation of $p$
  induced on~$M$ can be represented as follows:
  \[
  p(\wec{x}^1,\ldots,\wec{x}^\ell)=
  \begin{pmatrix}
    p_1(x_1^1,x_1^2,\ldots,x_i^j,\ldots,x_k^\ell)\\
    \vdots\\
    p_k(x_1^1,x_1^2,\ldots,x_i^j,\ldots,x_k^\ell)
  \end{pmatrix}\,,
  \]
  where $\wec{x}^j$ is (the transpose of) $(x_1^j, \ldots,
  x_k^j)$ and each $p_i$ is a $k\ell$-ary operation of~$\m
  S|_N$. Hence each $p_i$ is essentially unary.

  We use that $p$ is idempotent. Consider
  $p(\wec{x},\ldots,\wec{x})$ and move the first coordinate
  $x_1$ of $\wec{x}$ in~$N$ while keeping all other
  coordinates fixed. Then the value of $p_1$ must change, so
  $p_1$ must depend on one of the variables
  $x_1^1,x_1^2,\ldots,x_1^\ell$, and therefore on no
  other variable. Let this variable be $x_1^j$. The fact
  that $p$~is idempotent shows that
  $p_1(x_1^1,x_1^2,\ldots,x_k^\ell)=x_1^j$. An analogous
  statement holds for all other rows, in particular, each
  $p_i$ is a projection, hence it is idempotent. Substitute
  $b=(b_1,\ldots,b_k)$ to the $j$-th variable of $p$ and
  $a=(a_1,\ldots,a_k)$ to all other variables. Then in the
  first row we get $b_1$, and in the $i$-th row we get
  $a_i=b_i$ (if $i\ge 2$). Therefore the result of the
  operation is $b$, and so there is an edge $(a,b)$ in
  $\Tr(p)$.

  If the type of~$\m S$ is $\atyp$, then $\m S|_N$ is
  polynomially equivalent to a $1$-dimensional vector space
  over a field~$\m F$, and $(\m  S|_N)^{[k]}$ 
  is polynomially equivalent to the module $N^k$ over $\m
  F^{k\times k}$, where the matrices act on $N^k$ by
  multiplication. Since $p$ is idempotent,
  \[
  p(\wec{x}^1,\ldots,\wec{x}^\ell)=
  M_1\wec{x}^1+\ldots+M_\ell\,\wec{x}^\ell
  \]
  holds for some matrices $M_i\in \m F^{k\times k}$
  satisfying that their sum is the identity matrix.

  Notice that for any $\wec{x},\wec{y}\in N^k$ we have
  \[
  M_1(\wec{x}+\wec{y})+M_2\,\wec{x}+\ldots+M_\ell\,\wec{x}=
  \wec{x}+M_1\wec{y}\,,
  \]
  so $(\wec{x},\wec{x}+M_1\wec{y})$ is an edge of $\Tr(p)$,
  and similarly, $(\wec{x},\wec{x}+M_i\,\wec{y})$ is an
  edge, too, for every $1\le i\le \ell$. Now let
  $\wec{a},\wec{b}\in N^k$.  Since
  \[
  M_1(\wec{b}-\wec{a})+\ldots+M_\ell(\wec{b}-\wec{a})=\wec{b}-\wec{a}\,,
  \]
  we have 
  \[
  \wec{a}+M_1(\wec{b}-\wec{a})+\ldots+M_\ell(\wec{b}-\wec{a})=\wec{b}\,,
  \]
  and by our remark, this yields a path of $\ell$ edges in
  $\Tr(p)$ connecting $\wec{a}$ to~$\wec{b}$.
\end{proof}

\begin{proof}[Proof of Theorem~\ref{solvcharthm}]
Let $A$ be a finite
solvable algebra. We argue by induction on~$|A|$. Let $U$ be a
neighborhood of~$\m A$. Then the induced algebra~$\m A|_U$ is
solvable, too, so we can assume that $U=A$. Choose and fix
an idempotent polynomial~$p$ of~$\m A$. Let $\alpha$ be a
minimal congruence of~$\m A$. By the induction assumption,
$\m A/\alpha$ is strongly connected with respect to
$p/\alpha$. Thus if $a,b\in A$ are given, then there exist
edges $(u_1,v_1), \ldots, (u_k,v_k)$ of $\Tr(p)$ such that
$a\equiv_\alpha u_1$, $v_k\equiv_\alpha b$, and each
$v_i\equiv_\alpha u_{i+1}$ for $1\le i\le k-1$. Therefore it
is sufficient to prove that each $\alpha$-class is strongly
connected. Let $S$ be such a class. Then the induced algebra
$\m A|_S$ is simple, abelian, and $p|_S$ is a polynomial of
this induced algebra, since $p$ is idempotent. Therefore
Lemma~\ref{transl} shows that $S$ is strongly connected indeed.

For the converse, suppose that $A$ is not solvable. Then
there is a prime quotient $\langle\delta,\theta\rangle$ in
$\Con(\m A)$ of nonabelian type. Let $U$ be a
$\langle\delta,\theta\rangle$-minimal set. By 
\cite{hobby-mckenzie}, Lemma~2.10 and
Theorems~4.15, 4.17, and 4.23, $U$ is a neighborhood,
and there is a pseudo-meet operation~$p$ on~$\m A|_U$. This
operation is idempotent, satisfies the identity
$p(x,y)=p\big(x,p(x,y)\big)$, and there is an element $1\in
U$ such that $p(1,x)=p(x,1)=x$ for every $x\in U$. These
properties imply that there is no nontrivial edge in
$\Tr(p)$ whose endpoint is~$1$, hence this graph
is not strongly connected. Indeed, if $p(c,d)=1$, then
\[
1=p(c,d)=p\big(c,p(c,d)\big)=p(c,1)=c\,,
\]
and then $1=p(c,d)=p(1,d)=d$, so $c=d=1$. This calculation
shows that any edge terminating at $1$ must originate from $1$.
Thus the proof of
Theorem~\ref{solvcharthm} is complete.
\end{proof}

\section{Solvable algebras with a pointed cube term}\label{solv}

\begin{thm}\label{1-solv}
  Let $\m A$ be a finite, nontrivial solvable algebra that
  has a pointed cube term. Then $d_{\m a}(n)\in
  \Theta(n)$. In fact, the algebra $\m A^n$ is generated by
  those elements of $A^n$ that are constant with the
  possible exception of one component.
\end{thm}

(This establishes for solvable algebras
the extra implication (ii)$\Rightarrow$(iv)
among the conditions of Theorem~\ref{basic}.)

\begin{proof}
  The quotient modulo any maximal congruence of $\m A$ is
  abelian, so the combination of
  Corollary~\ref{abelian_cor}~(2) and
  Theorem~\ref{basic_estimates}~(2) shows that $d_{\m
    a}(n)\in \Omega(n)$. Therefore it is sufficient to prove
  that $d_{\m a}(n)\in O(n)$.

  For a given~$n$, let $G$ consist of those elements of
  $A^n$ that are constant with the possible exception of one
  component. The size of~$G$ is $|A|+n|A|(|A|-1)$ for $n\ge
  3$, which is linear in $n$. Thus it is enough to prove
  that $G$ is a generating set of~$\m A^n$. Let $\m B$ be
  the subalgebra generated by~$G$. Note that $B$ is
  symmetric under any permutation of coordinates. This will
  allow us to simplify notation in the proof below by
  rearranging coordinates at certain points.

  We may (and shall) assume that our cube identities involve
  only the variable~$x$, since all other variables can be
  replaced by a fixed constant from~$A$.

  We induct on $|A|$. Let $\alpha$ be a minimal congruence
  of~$\m A$. By our induction hypothesis, $(A/\alpha)^n$ is
  generated by $G/(\alpha^n)$, since this is the set of all
  sequences in $(A/\alpha)^n$ that are constant with the
  possible exception of one component. Therefore $B$
  intersects each $\alpha^n$-class. Our proof shall proceed
  by ``repairing'' these representatives component by
  component.

  For a given $1\le m\le n$ call a sequence
  $(b_1,\ldots,b_m)$ of elements of~$A$ \emph{complete}, if
  for every $a_{m+1},\ldots,a_n$ there exist elements
  $b_{m+1},\ldots,b_n$ such that $(b_1,\ldots,b_n)\in B$ and
  $b_i\equiv_\alpha a_i$ for all $i>m$. We shall say
  informally that \emph{$(b_1,\ldots,b_m)$ can be extended
    into~$B$ along $a_{m+1},\ldots,a_n$.} The argument in
  the previous paragraph shows that the empty sequence is
  complete (when $m=0$), and we endeavor to show that every
  element of $A^n$ is complete, which means that
  $B=A^n$. Therefore it is sufficient to prove the
  following.

  \begin{clm}\label{1-solvclm}
    Suppose that each element of $A^{m-1}$ is complete. Then
    each element of $A^m$ is complete.
  \end{clm}

\cproof
  To set up notation, we permute the coordinates in the
  following way. We assume that $(b_2,\ldots,b_m)$ is
  complete and intend to show that $(c,b_2,\ldots,b_m)$ is
  complete for every~$c$.

  Fix $a_{m+1},\ldots, a_n\in A$. We have to extend
  $(c,b_2,\ldots,b_m)$ into $B$ along $a_{m+1},\ldots,
  a_n$. The completeness of $(b_2,\ldots,b_m)$ allows us to
  find elements $b,b_{m+1},\ldots,b_n$ such that
  $b\equiv_\alpha c$,\; $b_j\equiv_\alpha a_j$ for all $j>m$
  and $\wec{b}=(b_1,\ldots,b_n)\in B$. In other words,
  $(b,b_2,\ldots,b_m)$ can be extended into $B$ along
  $a_{m+1},\ldots, a_n$.

  Let $S=b/\alpha=c/\alpha$. The induced algebra $\m S=\m
  C|_S$ is simple and abelian. Consider a cube identity
  $F(x,\ldots, x,u_1,\ldots,u_p)\approx x$ (here we
  rearranged the variables appropriately and
  $u_1,\ldots,u_p$ are constants). For each constant $u_j$
  extend $(b_2,\ldots,b_m)$ into~$B$ along $u_j$ and
  $a_{m+1},\ldots, a_n$. We get a vector $\wec{u}^j\in B$
  whose first coordinate is denoted by~$u_j'$.

  Consider the polynomials 
  \begin{align}\label{pq}
    p(x_1,\ldots,x_\ell) &= F^{\m A}(x_1,\ldots,
 x_\ell,u_1,\ldots,u_p)\\
    q(x_1,\ldots,x_\ell) &= F^{\m A}(x_1,\ldots,
    x_\ell,u_1',\ldots,u_p')\,.\notag
  \end{align}
  The first polynomial is idempotent because of the cube
  identity above. Therefore $p(S,\ldots,S)\subseteq S$, and
  $p$~restricted to~$S$ is a polynomial of~$\m S$. Since
  $u_j\equiv_\alpha u_j'$, the polynomial~$q$ can also be
  restricted to~$S$.

  The congruence~$\alpha$ is abelian. This implies that
  $\pi(x)=q(x,x,\ldots,x)$ is a permutation of~$S$. Indeed,
  if $q(s,\ldots s)=q(t,\ldots,t)$
  for some $s,t\in S$, then using the term
  condition for $F$ (which can be applied, since
  $s\equiv_\alpha t$ and
  $u_j\equiv_\alpha u_j'$) we get that $p(s,\ldots, s)
  =p(t,\ldots,t)$. Since~$p$ is idempotent, we see that
  ${s=t}$. Thus $\pi$ is indeed a permutation. Let $o$ be
  the order of $\pi$ in the symmetric group on~$S$ and
  $r(x_1,\ldots,x_\ell) =
  \pi^{o-1}\big(q(x_1,\ldots,x_\ell)\big)$. This is an
  idempotent polynomial of $\m S$.

  Lemma~\ref{transl} can be applied to the polynomial~$r$ to
  connect $b$ to~$c$ by edges in $\Tr(r)$. We plan to go
  along this path and extend into~$B$ step by step, so it is
  enough to show how to do a single edge. For the simplicity
  of notation we may assume that $(b,c)$ itself is an
  edge. That is, also by appropriately rearranging the
  arguments of~$F$, that $r(d,b,\ldots,b)=c$ for some
  element~$d\in S$. This means that
  $q(d,b,\ldots,b)=\pi(c)$.

  Now choose a second cube identity, where the first
  variable is not~$x$. We shall apply the term~$F$ to a
  matrix~$M$ whose elements we now describe. The number of
  rows is~$n$ and the number of columns is the arity
  of~$F$. We shall distinguish five types of columns of~$M$,
  that is, arguments of~$F$. The definitions of these types
  and the way we fill the entries of the matrix~$M$ are
  shown on Figure~\ref{FigI-V} (page~\pageref{FigI-V}).

\newmdenv[linewidth=1pt]{mybox}

  \begin{figure}[ht]
\begin{mybox}
    \leftline{The types of columns of~$M$:}
  \begin{enumerate}
  \item[(I)] The first column/argument is the only one to have
    type~I. There is $x$ in the first cube identity and a
    constant named $s$ in the second cube identity.
  \item[(II)] Type II column/argument: both cube
    identities have~$x$ in this argument.
  \item[(III)] Type III: there is $x$ in the first cube
    identity, and a constant in the second. We exclude
    the first argument (which also has this
    property).
  \item[(IV)] Type IV: there is a constant in the first
    cube identity, and $x$ in the second.
  \item[(V)] Type V: there are constants in both cube
    identities.
  \end{enumerate}
  \leftline{The way to fill the columns of~$M$:}
  \begin{enumerate}
  \item[(V)] In a type V argument, let $u_k$ be the constant
    in the first cube identity, and $w$ the constant in the
    second cube identity. Then put $u_k'$ (defined above)
    into the first coordinate and $w$ to all other
    coordinates. (Of course for different columns these
    constants may differ.) This column is in $G$.
  \item[(IV)] Into a type IV column, put the appropriate
    vector $\wec{u}^j$, also defined above. Recall that this
    vector extends $(b_2,\ldots,b_m)$ along
    $u_j,a_{m+1},\ldots,a_n$ into $B$, and its first entry
    is $u_j'$.
  \item[(III)] For a type III argument, let $t$ be the
    constant appearing in the second cube identity. Put $b$
    to the first coordinate and $t$ everywhere else. This
    column is in~$G$.
  \item[(II)] Put the vector $\wec{b}\in B$ defined above to all
    type II columns.
  \item[(I)] To the first column (the only column of
    type~$I$), put the element~$d$ to the first coordinate,
    and $s$ (the constant in the first variable of the
    second cube identity) everywhere else. This is also
    in~$G$.
  \end{enumerate}
  \[
  \begin{matrix}
    \text{I} & \text{II} & \text{III} & \text{IV} & \text{V}\\
    d & b & b & u_j' & u_k'\\
    s & b_2 & t & b_2 & w\\
    s & b_3 & t & b_3 & w\\
    \vdots &\vdots &\vdots &\vdots &\vdots\\
    s & b_m & t & b_m & w\\
    s & b_{m+1} & t & b_{m+1}' & w\\
    \vdots &\vdots &\vdots &\vdots &\vdots\\
    s & b_n & t & b_n' & w
  \end{matrix}
  \]
\end{mybox}
  \caption{The elements substituted into the five types of columns}\label{FigI-V}
  \end{figure}

  Applying $F$ to this matrix the first row yields
  $q(d,b,\ldots,b)=\pi(c)$. All other rows can be evaluated
  using the second cube identity. The result is $b_j$ for
  $2\le j\le m$ and is $\alpha$-related to $a_j$ for $j>m$,
  since $b_j, b_j'\equiv_\alpha a_j$.

  This is almost the vector we are looking for, we only have
  to replace $\pi(c)$ by $c$ in the first component. In
  other words, it is sufficient to prove the following.

  \begin{subclm}\label{deletepi}
    Suppose that $(\pi(c_1),c_2,\ldots,c_n)\in B$ holds for
    some vector where $c_1\in S$. Then
    $(c_1,c_2,\ldots,c_n)\in B$.
  \end{subclm}

  \subcproof The equations in display (\ref{pq}) and
  $\pi(x)=q(x,x,\ldots,x)$ show that we have
  $\pi^{o-1}(x)=g(x,u_1',\ldots,u_p')$ for a term $g$
  obtained from $F$ by composition, and if we replace each
  $u_i'$ by $u_i$, then we get that $g(x,u_1,\ldots,u_p)=x$
  (the $(o-1)$st power of the identity map). The columns of
  the matrix
  \[
  \begin{pmatrix}
    \pi(c_1) & u_1' & \ldots & u_p'\\
    c_2 & u_1 & \ldots & u_p\\
    \vdots &\vdots &\vdots &\vdots\\
    c_n & u_1 & \ldots & u_p\\
  \end{pmatrix}
  \]
  are in $B$ (for the first column we assumed this, and the
  others are in~$G$). Applying the term~$g$ to the rows we
  get that $(c_1,c_2,\ldots,c_n)\in B$, proving the subclaim and
  Claim~\ref{1-solvclm}. This completes the proof of
  Theorem~\ref{1-solv}, too.
\end{proof}

\section{Spreads and growth rate}\label{spreadsec}

\begin{df}\label{spread}
  Let $\m A$ be an algebra and $\mathcal U$ a collection of
  subsets of~$A$. We say that a subset $S\subseteq A$ is a
  \emph{spread} of~$\mathcal U$ if there exists a
  polynomial~$p$ of~$\m A$ and (not necessarily distinct)
  elements $U_1,\ldots,U_k\in\mathcal U$ such that
  $p(U_1,\ldots,U_k)=S$.
\end{df}

\begin{lm}\label{spreadaffrate}
  Suppose that $\m A$ is a finite algebra such that all
  constants are term operations, $p$ is a polynomial of~$\m
  A$ and $A=p(U_1,\ldots, U_k)$ holds for some subsets
  $U_i\subseteq A$. Then
  \[
  d_{\m A}(n)\le d_{\m U_1}(n)+\ldots+d_{\m U_k}(n)\,,
  \]
  where $\m U_i=\m A|_{U_i}$ are the algebras 
  $\m a$ induces on these sets.
\end{lm}

\begin{proof}
  For a given~$n$ let $G_i$ be a smallest size generating set of the
  algebra $\m U_i^n$. The set $G:=\bigcup_{i=1}^k G_i$ has size
  at most $d_{\m U_1}(n)+\ldots+d_{\m U_k}(n)$, and generates
  a subalgebra of $\m a^n$ containing all sets of the form 
  $U_i^n$. Since $p$ is a term operation of $\m a$, and 
\[
p^{\m a^n}(U_1^n,\ldots,U_k^n)=\left(p^{\m a}(U_1,\ldots,U_n)\right)^n = A^n
\]
we get that $G$ generates $\m a^n$.
\end{proof}

\begin{cor}\label{spreadafflin}
  If a finite algebra $\m A$ is a spread of a
  family of subsets on which the induced algebras have Maltsev polynomials,
  then the growth rate of~$\m A$ is at most linear.
\end{cor}

\begin{proof}
  Use Corollary~\ref{pointed_polynomial_cor} and Lemma~\ref{spreadaffrate}.
\end{proof}

Several years ago,
K.~A.~Kearnes, E.~W.~Kiss and
M.~A.~Valeriote proved but did not publish
the fact that any finite algebra with a Maltsev polynomial
is \emph{covered} by its $\langle\alpha,\beta\rangle$-minimal sets. 
This statement
is a bit stronger than the statement that 
any finite algebra with a Maltsev polynomial
is a spread of its minimal
sets, so it is relevant here. 

To introduce the terminology, 
if $\mathcal U$ and $\mathcal V$
are sets of neighborhoods of $\m a$, then 
$\mathcal U$ \emph{covers} $\mathcal V$ if for
every $V\in {\mathcal V}$ there are
neighborhoods $U_1,\ldots,U_k\in {\mathcal U}$, idempotent
unary polynomials $e_1,\ldots, e_k, f$, and other
polynomials $\lambda, \rho_i$ such that $f(A)=V$,
$e_i(A)=U_i$ and
\begin{equation}\label{decomposition}
\lambda\big(e_1\rho_1(x),\ldots,e_k\rho_k(x)\big)=f(x)
\end{equation}
for all $x\in A$. (Here the $e_i$'s need not be distinct.)
Equation (\ref{decomposition}) expresses $V$
as a retract of the product of the $U_i$'s via
polynomial maps $\lambda(x_1,\ldots,x_k)$ and
$\big(e_1\rho_1(x),\ldots,e_k\rho_k(x)\big)$.
It is not hard to see that if $\mathcal U$ covers
$\mathcal V$ and $\mathcal V$ covers $\mathcal W$,
then $\mathcal U$ covers $\mathcal W$.
We say that $\mathcal U$ covers the algebra $\m a$ if $\mathcal U$
covers the set $\{A\}$.

Lemma~2.10 of \cite{hobby-mckenzie}
guarantees that the minimal sets of a finite algebra are neighborhoods.
If some set $\mathcal U$ of minimal sets covers $\m a$, 
and this is witnessed as in (\ref{decomposition})
by polynomials $e_i, \lambda, \rho_i$ and $f={\rm id}$,
then equation (\ref{decomposition}) yields 
$\lambda(U_1,\ldots,U_k)=A$, so $A$ is a spread of the 
minimal sets in $\mathcal U$.
Hence if $\m a$ is covered by its minimal sets, then
it is a spread of its minimal sets. 

\begin{thm}\label{kkv}
  If $\m a$ is a finite algebra with a Maltsev polynomial,
  then $\m a$ is covered by its
  $\langle\alpha,\beta\rangle$-minimal sets, where
  $\langle\alpha,\beta\rangle$ runs over all prime quotients
  of\/~$\Conb(\m a)$.
\end{thm}

\begin{proof}
Let $N\subseteq A$ be a neighborhood of
$\m a$ that is maximal for the property that
$\m a|_N$ is covered by its minimal sets.
If $N=A$, then we are done, so we assume otherwise
and argue to a contradiction.

By Theorem~4.31 of~\cite{hobby-mckenzie}, type~$\atyp$
minimal sets in an algebra with a Maltsev polynomial are
E-minimal. Hence a minimal set of a finite algebra with a
Maltsev polynomial is nothing other than a set that is
minimal under inclusion among nontrivial
neighborhoods. Therefore the minimal sets of $\m a|_N$ are
exactly the minimal sets of $\m a$ that lie in $N$.

It is a basic fact of tame congruence theory (a consequence
of Lemma~2.17 of~\cite{hobby-mckenzie}) that every finite
algebra is the connected union of the traces of its minimal
sets for congruences, hence if $N\neq A$ then there must
exist $\alpha\prec\beta$ in $\Conb(\m a)$ and an
$\lb\alpha,\beta\rb$-minimal set $U$ such that $U$ has a
trace $T\subseteq U$ that properly overlaps $N$.  (I.e.,
$T\cap N\neq\emptyset$, but $T\not\subseteq N$.)  Choose
$0\in T\cap N$. Thus $U$ properly overlaps $N$ and $U\cap N$
contains an element $0$ from (the body of) $U$.  Let
$m(x,y,z)$ denote a Maltsev polynomial of $\m a$.

\begin{clm}
$\m a$ has idempotent unary polynomials $e$
and $f$ such that 
\begin{enumerate}
\item 
$e(A) = U$ and $e(N) = \{0\}$, and
\item
$f(A) = N$ and $f(U) = \{0\}$.
\end{enumerate}
\end{clm}

Since $U$ and $N$ are neighborhoods there exist
idempotent unary polynomials $e_0$ and $f_0$
such that $e_0(A)=U$ and $f_0(A)=N$. We show that these
can be modified to have the extra properties in the claim.

Our first step is to construct $e$.  As a first case, assume
that $e_0f_0\in \Pol_1(\m a)|_U$ is not a permutation of
$U$. The polynomial $g(x)=e_0m\big(0,e_0f_0(x),e_0(x)\big)$
has the property that $g(A)\subseteq U$ and
$g(N)=\{0\}$. Take $u \in T$ such that $(u,0) \in \beta -
\alpha$. Then $\big(e_0f_0(u),0\big)\in \alpha$, and we get
$\big(g(u),u\big)\in \alpha$. Since $(0,u)\in\beta-\alpha$,
$g(0)=0$, and $g(u)\equiv u\pmod{\alpha}$ it follows that
$g(\beta)\not\subseteq \alpha$ and so $g$ is not collapsing
on $U$. Therefore an appropriate power $e = g^k$ is an
idempotent unary polynomial with range $U$ that collapses
$N$ to $\{0\}$.

Now assume that $e_0f_0$ is a permutation of $U$.
Then $U\simeq f_0(U)$, so $V := f_0(U)\subseteq N$
is an $\lb\alpha,\beta\rb$-minimal set contained in $N$ and containing $0$.
Let $S=f_0(T)$ be the $\lb\alpha,\beta\rb$-trace of $V$ containing $0$.
Corollary~4.8 of \cite{kkv2} 
guarantees that there is an idempotent unary polynomial $e_1$
of $\m a$ such that in the quotient 
$\m a/\alpha$
we have $\overline{e}_1(A/\alpha) = U/\alpha$ and 
$\overline{e}_1(S/\alpha) = \{0/\alpha\}$.
Replacing $e_1$ by $e_0e_1$ if necessary
we may assume that $e_0e_1 = e_1$, so $e_1(A)\subseteq U$.
Since $e_1$ maps $A$ into $U$ and it
is the identity modulo $\alpha$ on $U$, it
follows from the $\lb\alpha,\beta\rb$-minimality of $U$
that $e_1(A)=U$.
Now $e_1f_0$ is collapsing on~$T$, so $e_1f_0$
is not a permutation of $U$.
Thus we can repeat the argument of the first case
using~$e_1$ in place of $e_0$ to construct
an idempotent unary polynomial $e$
such that $e(A)=U$ and $e(N)=\{0\}$.

Now, given $e$ and $f_0$ as above 
let $f(x) = f_0m\big(f_0(x),m\big(e(x),x,f_0(x)\big),0\big)$.
One calculates that $f(A)\subseteq N$, $f$ is the identity
on $N$ (so $f$ is idempotent with range $N$),
and $f(U)=\{0\}$.
This completes the proof of the claim.\cqed

\medskip

Now we finish the proof of the theorem.
The polynomial $h(x)=m\big(e(x),0,f(x)\big)$ is the identity
on $N\cup U$, so some iterate $h^{\ell}(x)$
is idempotent with range $h^{\ell}(A)=M\supsetneq N$.
The equation 
\[
h^{\ell}m\big(e\big(h^{\ell-1}(x)\big),0,
f\big(h^{\ell-1}(x)\big)\big)=h^{2\ell}(x)=h^{\ell}(x)
\]
is an equation of type (\ref{decomposition})
for $\m a|_M$, showing that $\m a|_M$ is covered by its subneighborhoods
$N$ and $U$. Because the covering relation is transitive,
$\m a|_M$ is covered
by~$U$ together with the minimal sets in $N$
which cover $\m a|_N$. This shows that $\m a|_M$ is covered by the minimal
sets it contains, contradicting the maximality assumption on~$N$.
\end{proof}

\begin{cor}\label{kkv_cor}
If\/ $\m a$ is a finite solvable algebra with a
Maltsev polynomial, then $\m a$~is a spread of its
type $\atyp$ minimal sets.
\end{cor}

(This establishes for solvable algebras
the extra implication (i)$\Rightarrow$(iii)
among the conditions of Theorem~\ref{basic}.)

\begin{proof}
Theorem~\ref{kkv}
proves that $\m a$ is a spread of its minimal sets.
They all have type~$\atyp$, since $\m a$ is solvable
and has a Maltsev polynomial.
\end{proof}

\section{Left nilpotent algebras}\label{abel}

In the papers \cite{sym}, \cite{strnil}, \cite{nilpwab}
various forms of nilpotence are discussed. An algebra~$\m A$
is left nilpotent if $[1_{\m A},\dots,[1_{\m A},[1_{\m
  A},1_{\m A}]]\dots]=0_{\m A}$ for a sufficiently long
expression. This definition implies that the class of left
nilpotent algebras is closed under subalgebras and finite
direct products.

\begin{thm}\label{nilpotent}
  A finite left nilpotent algebra has a Maltsev polynomial
  iff it has a pointed cube polynomial.
\end{thm}

(This establishes for nilpotent algebras
the extra implication (ii)$\Rightarrow$(i)
among the conditions of Theorem~\ref{basic}.)

\begin{proof}
  A Maltsev polynomial is a $3$-ary, $0$-pointed, $2$-cube
  polynomial. To prove the converse, we can assume that all
  elements of~$A$ are the interpretations of nullary operation
  symbols. Thus
  if $\m A$ has a pointed cube polynomial, then it is in
  fact a pointed cube term, and so it is interpreted as a
  pointed cube term in every algebra of the variety~$\vr V$
  generated by~$\m A$. Thus, Theorem~\ref{pointed_polynomial} implies
  that no algebra in~$\vr V$ can have exponential growth
  rate. However, nontrivial strongly abelian algebras have
  exponential growth rate by
  Corollary~\ref{abelian_cor}. Therefore no such algebra
  exists in~$\vr V$. We finish the proof by recalling
  Theorem~6.8 of~\cite{kkv1}, which states that if a
  variety~$\vr V$ is generated by a left nilpotent algebra,
  and there is no nontrivial strongly abelian algebra
  in~$\vr V$, then $\m A$ has a Maltsev polynomial.
\end{proof}

The following corollary follows from the fact
that if an abelian algebra has a Maltsev polynomial, then it
has a Maltsev term.

\begin{cor}\label{affine_cor}
  A finite abelian algebra has a pointed cube polynomial iff
  it is affine.
\qed
\end{cor}

Our next goal is to prove the following theorem.

\begin{thm}\label{spreadnilp}
  If $\m A$ is a finite, left nilpotent algebra, and $\m
  A^{|A|}$ does not have a nontrivial strongly abelian
  quotient algebra, then $\m A$ is a spread of its
  type~$\atyp$ minimal sets.
\end{thm}

(This establishes for nilpotent algebras
the extra implication (vi)$\Rightarrow$(iii)
among the conditions of Theorem~\ref{basic}.)

\begin{lm}\label{proj}
  Let $\m A$ be a finite solvable algebra and
  $\eta_1,\ldots,\eta_n$ congruences of~$\m A$. Then for
  each type~$\atyp$ covering~$\delta\prec\theta$ in $\Con(\m
  A)$ such that $\bigwedge_{i=1}^n \eta_i\le\delta$ there is
  an $i$ and congruences $\eta_i\le\alpha\prec\beta$ such
  that $\lb\delta,\theta\rb$ and $\lb\alpha,\beta\rb$ are
  projective in~$\Con(\m A)$ (and hence have the same
  minimal sets).
\end{lm}

\begin{proof}
  Let $\gamma$ be an arbitrary congruence of~$\m A$. We
  prove that either
  $\lb\gamma\vee\delta,\gamma\vee\theta\rb$ or
  $\lb\gamma\wedge\delta,\gamma\wedge\theta\rb$ is
  perspective to~$\lb\delta,\theta\rb$. Indeed, since
  $\delta\prec\theta$, it is sufficient to prove that
  $\theta\wedge(\gamma\vee\delta)\ne\theta$ or
  $\delta\vee(\gamma\wedge\theta)\ne\delta$. Suppose that
  both inequalities fail.
  The quotient lattice of $\Conb(\m A)$
  modulo the strongly solvability congruence is modular by
  Theorem~7.7~(4) of~\cite{hobby-mckenzie}. Therefore
  $\theta=\theta\wedge(\gamma\vee\delta)$ and
  $\delta=\delta\vee(\gamma\wedge\theta)$ are related by the
  strongly solvability congruence. This contradicts the
  fact that $\lb\delta,\theta\rb$ has type~$\atyp$.

  Now apply this observation to $\gamma=\eta_1$. Either the
  statement of the lemma is satisfied with $i=1$, or
  $\lb\eta_1\wedge\delta,\eta_1\wedge\theta\rb$ is
  perspective to~$\lb\delta,\theta\rb$. Let
  $\delta_1=\eta_1\wedge\delta$ and $\theta_1$ be any upper
  cover of $\delta_1$ that is below
  $\eta_1\wedge\theta$. Then $\lb\delta_1,\theta_1\rb$ is
  still perspective to $\lb\delta,\theta\rb$, since
  $\delta\vee\theta_1=\theta$, but it is now a prime
  quotient. Perspective prime quotients have the same type,
  so $\lb\delta_1,\theta_1\rb$ also has type~$\atyp$.

  Now apply the previous observation to $\eta_2$ and this
  new quotient. Again, either the statement of the lemma
  holds for $i=2$, or we can push our cover down below
  $\eta_2$. Continuing this process, if the statement of the
  lemma fails for every~$i$, then we get a prime quotient
  $\lb\delta_n,\theta_n\rb$ that is still projective
  to~$\lb\delta,\theta\rb$, and
  $\delta_n=\big(\bigwedge_{i=1}^n\eta_i\big)\wedge\delta$,
  while $\theta_n\le
  \big(\bigwedge_{i=1}^n\eta_i\big)\wedge\theta$.  But our
  assumption that $\bigwedge_{i=1}^n \eta_i\le\delta$
  implies that $\theta_n=\delta_n=\bigwedge_{i=1}^n \eta_i$,
  which is a contradiction.
\end{proof}

\begin{cor}\label{type2}
  Let $\m A$ be a finite solvable algebra and $\mathcal U$ a
  set containing exactly one member from each polynomial
  isomorphism class of type $\atyp$ minimal sets of $\m A$.
  If\/ $\delta\prec\theta$ is a covering of type $\atyp$ in
  $\Conb(\m A^n)$, then there is a $\lb
  \delta,\theta\rb$-minimal set of the form $U^n$ for some
  $U\in \mathcal U$.
\end{cor}

\begin{proof}
  Denote by $\eta_1,\ldots,\eta_n$ the projection kernels
  of~$\m A^n$, and apply Lemma~\ref{proj} to this set of
  congruences on~$\m A^n$. We get that the covering
  $\delta\prec\theta$ is projective to a covering
  $\alpha\prec\beta$ which lies above some $\eta_i$.  Thus
  the $\lb\delta,\theta\rb$-minimal sets are the same as the
  $\lb\alpha,\beta\rb$-minimal sets.  Identifying $\m a$
  with $\m A^n/\eta_i$, choose some $U\in\mathcal U$ that is
  an $\lb \alpha/\eta_i,\beta/\eta_i\rb$-minimal set.  Then
  (a) $U^n$ is the image of an idempotent unary polynomial
  of~$\m a^n$, (b) $\m a^n|_{U^n}$ is an E-minimal algebra
  of type $\atyp$ (since $\m a^n|_{U^n}$ is polynomially
  equivalent to $(\m a|_U)^n$ and powers of solvable
  E-minimal algebras are E-minimal, according to Lemma 4.10
  of \cite{sym}), and (c) $\alpha|_{U^n}\neq \beta|_{U^n}$.
  Items (a)--(c) are enough to show that $U^n$ is a minimal
  set for $\lb \alpha,\beta\rb$ and hence for
  $\lb\delta,\theta\rb$.
\end{proof}

It is proved in~\cite{sym} that left nilpotence implies the
following, weaker condition:
\begin{equation}\label{tracenilp}
\text{$\NC(1_{\m A},N^2;\delta)$
holds whenever $\delta\prec\theta$
and $N$ is a $\lb\delta,\theta\rb$-trace.}
\tag{$\dag$}
\end{equation}
This condition is clearly still stronger than
solvability. One of the main results of~\cite{ham} is that
in an algebra satisfying~(\ref{tracenilp}), every maximal
subalgebra is a block of a congruence. We shall use this
result in the following proof.

\begin{proof}[Proof of Theorem~\ref{spreadnilp}]
  Let $\mathcal U=\{U_1, \ldots, U_k\}$ be a set containing
  exactly one member from each polynomial isomorphism class
  of type $\atyp$ minimal sets of $\m A$. Let $n=|A|$ and let
  $\m B$ be the subalgebra of $\m A^n$ generated by all sets of
  the form $U^n$, where $U\in\mathcal U$. Thus $B$ is the
  union of the sets $t^{\m A^n}(U_1^n,U_2^n,\ldots,U_k^n)$,
  where $t$ is a term of~$\m A$. Clearly 
  $t^{\m A^n}(U_1^n,U_2^n,\ldots,U_k^n) = 
  \left(t^{\m A}(U_1,U_2,\ldots,U_k)\right)^n$.

  There are two cases. If $B=A^n$, then let $\wec{a}$ be a
  listing of all elements of~$\m A$. Then $\wec{a}\in
  A^n=B$, and so there exists a term $t$ such that
  $\wec{a}\in t^{\m A}(U_1,U_2,\ldots,U_k)^n$. Then $t^{\m
    A}(U_1,U_2,\ldots,U_k)=A$, and so $\m A$ is a spread
  of~$\mathcal U$.

  In the other case, $\m B$ is a proper subalgebra of~$\m
  A^n$. Let $\m M$ be a maximal subalgebra of~$\m A^n$
  containing~$B$. Since $\m A^n$ is left nilpotent, the
  result quoted above implies that $M$ is a block of some
  congruence~$\mu$ of~$\m A^n$. We shall prove that $\m
  A^n/\mu$ is strongly solvable. This is sufficient, since
  then the quotient modulo any maximal congruence
  containing~$\mu$ is strongly abelian.

  Suppose that there is a prime
  quotient~$\mu\le\delta\prec\theta$ of
  type~$\atyp$. Corollary~\ref{type2} states that some
  $U_i^n$ is a minimal set for $\lb\delta,\theta\rb$. This
  is a contradiction, since $U_i^n\subseteq B\subseteq M$ is
  contained in a single $\mu\le\delta$-class. Thus the proof
  of~Theorem~\ref{spreadnilp} is complete.
\end{proof}

\section{Abelian varieties}\label{abvar}

We show that in an abelian variety, any algebra that is a
spread of affine subsets is affine. Note that if $\m A$ is a
finite abelian algebra, $\lb\alpha,\beta\rb$ is a prime
quotient of type~$\atyp$, and $U\in M_{\m A}(\alpha,\beta)$,
then the induced algebra~$\m A|_U$ is affine. Indeed, $\m A$
is solvable, so Lemma~4.27~(4) of \cite{hobby-mckenzie}
implies that the tail of~$U$ is empty. By Theorem~4.31
of~\cite{hobby-mckenzie}, the induced algebra on~$U$ is
Maltsev (and E-minimal), and as $\m A$ is abelian, it is
affine.

Recall that an abelian group operation on an algebra $\m A$
is \textsl{compatible}, if adding $x-y+z$ as a term makes
$\m A$ an affine algebra.

\begin{thm}\label{spreadaff}
  Let $\m A$ be an algebra and $U_1,\ldots,U_k\subseteq A$
  such that $A=t(U_1,\ldots,U_k)$ for a polynomial~$t$
  of~$\m A$. If all induced algebras $\m A|_{U_i}$ are
  affine, then the following hold.
\begin{enumerate}
\item If\/ $\Ho(\m A^2)$ is abelian, then there is an abelian
  group operation~$+$ on~$A$ that is compatible with all
  operations of~$\m A$. Moreover, if $\alpha\in\Con(\m A)$,
  then $\alpha$ is a congruence of the group $(A, +)$.
\item If the variety ${\mathcal V}(\m A)$ generated by $\m A$ 
is abelian, then $\m A$ is affine.
\end{enumerate}
\end{thm}

(This establishes for algebras generating abelian varieties
the extra implication (iii)$\Rightarrow$(i) among the
conditions of Theorem~\ref{basic}).

\begin{proof}
  For each $1\le i\le n$ choose and fix an element $0_i\in
  U_i$, a binary polynomial $+_i$ and a unary polynomial
  $-_i$ of $\m A|_{U_i}$ such that $(U_i,+_i,-_i,0_i)$ is an
  abelian group.

  To define the operation $+$ on $A$ let $a,b\in A$. Then
  $a=t(\wec{a})$ and $b=t(\wec{b})$ for some $a_i,b_i\in
  U_i$. Define
  \[
  a+b=t(a_1 +_1 b_1,\ldots, a_k+_k b_k)\,.
  \]
  This operation is well-defined, as we now show. Assuming
  that $a=t(\wec{a'})$, we have to prove that $t(a_1 +_1
  b_1,\ldots,\allowbreak a_k+_k b_k)= t(a_1' +_1
  b_1,\ldots,\allowbreak a_k'+_k b_k)$.  This is an instance
  of the abelianness of $\m A$, because
  $t(\wec{a})=t(\wec{a'})$ implies $t(a_1 +_1 0_1,\ldots,
  a_k+_k 0_k)=t(a_1' +_1 0_1,\ldots, a_k'+_k 0_k)$. The
  argument is the same for the second variable of~$+$. It is
  clear that this operation is associative, commutative,
  $0=t(0_1,\ldots, 0_k)$ is a zero element, and
  $-t(\wec{a})=t(-_1 a_1,\ldots, -_k a_k)$ is the
  (well-defined) negative of $t(\wec{a})$.

  To prove that $+$ is compatible it is enough to show that
  the congruence~$\Delta$ of $\m A^2$ obtained by collapsing
  the diagonal satisfies $(a,b)\equiv_\Delta (c,d)
  \iff a+d=b+c$.

  Suppose first that $a+d=b+c$, we prove that
  $(a,b)\equiv_\Delta (c,d)$. Let $a=t(\wec{a})$,
  $b=t(\wec{b})$ and $c=t(\wec{c})$. Define $d_i=b_i +_i c_i
  -_i a_i$, clearly $d=t(\wec{d})$. We have
  $(a_i,a_i)\equiv_\Delta (c_i,c_i)$, and using the unary
  polynomial of $\m A^2$ that maps $(x,y)$ to $(x+_i
  0_i,y+_i b_i -_i a_i)$ we get that
  $(a_i,b_i)\equiv_\Delta (c_i,d_i)$. Applying $t$ we see
  that $(a,b)\equiv_\Delta (c,d)$ indeed.

  So far, we have only used that $\m A$ is abelian, now we
  need that $\m A^2/\Delta$ is abelian, too. Suppose that
  $(a,b)\equiv_\Delta (c,d)$. Let $a=t(\wec{a})$,
  $b=t(\wec{b})$, $c=t(\wec{c})$ and $d=t(\wec{d})$. We have
  that
  \[
  \begin{pmatrix}
  t(a_1+_1 0_1,\ldots, a_k+_k 0_k)\\
  t(b_1+_1 0_1,\ldots, b_k+_k 0_k)
  \end{pmatrix}
  \equiv_\Delta 
  \begin{pmatrix}
  t(c_1+_1 0_1,\ldots, c_k+_k 0_k)\\
  t(d_1+_1 0_1,\ldots, d_k+_k 0_k)
  \end{pmatrix}\,.
  \]
  Since $\m A^2/\Delta$ is abelian, we can replace each
  $0_i$ in the second row by $a_i-_i b_i$. Therefore
  \[
  \begin{pmatrix}
  t(a_1,\ldots, a_k)\\
  t(a_1,\ldots, a_k)
  \end{pmatrix}
  \equiv_\Delta 
  \begin{pmatrix}
  t(c_1,\ldots, c_k)\\
  t(d_1+_1 a_1-_k b_1,\ldots, d_k+_k a_k -_k b_k)
  \end{pmatrix}\,.
  \]
  That is, $(a,a)\equiv_\Delta (c,d+a-b)$. The abelianness
  of $\m A$ implies that the diagonal is a $\Delta$-class,
  hence $c=d+a-b$. Thus $+$ is indeed a compatible
  operation.

  To prove that every congruence $\alpha$ of $\m A$ is a
  group-congruence it is sufficient to verify that if
  $(a,b)\in\alpha$, then $(a+c,b+c)\in\alpha$. Again, let
  $a=t(\wec{a})$, $b=t(\wec{b})$ and
  $c=t(\wec{c})$. Then
  \[
  t(a_1+_1 0_1,\ldots, a_k+_k 0_k)
  \equiv_\alpha 
  t(b_1+_1 0_1,\ldots, b_k+_k 0_k)\,.
  \]
  Since $\m A/\alpha$ is abelian, we can move each $0_i$ to
  $c_i$, proving~(1).

  Now suppose that ${\mathcal V}(\m A)$ is abelian. Recall
  that an algebra is called \textsl{Hamiltonian,} if every
  subalgebra is a block of some congruence. By the main
  result of~\cite{kiss-val}, every member of ${\mathcal
    V}(\m A)$ is Hamiltonian. Notice that $A^n=\hat
  t(U_1^n,\ldots,U_k^n)$ holds in $\m A^n$, where $\hat t$
  is $t$ acting componentwise. Hence (1) applies to all
  powers of $\m A$, and the compatible group operation
  on~$\m A^n$ is the operation of the group
  $(A,+)^n$. Construct the free algebra $\m F$ of ${\mathcal
    V}(\m A)$ generated by $x$, $y$ and $z$ as a subalgebra
  of $\m A^n$, where $n=|A|^3$. By the Hamiltonian property,
  there exists a congruence $\alpha$ of $\m A^n$ such that
  $F$ is a class of $\alpha$. Item (1) of the theorem
  implies that $\alpha$ is a group congruence, so $F$ is a
  coset modulo a subgroup, hence $x-y+z\in F$. Therefore
  $x-y+z$ is a term operation of~$\m A$.
\end{proof}

We present two examples showing that no obvious weakening of
the conditions in the previous theorem is possible. 
Our starting point is the example following Corollary~4.4 of
\cite{kkv2}. 

\begin{exmp}
Let $\BZ_2$ be the two-element field, let
$A=\BZ_2^3$, denote by $+$ the (elementary abelian) group
operation on~$A$, and consider the following matrices in
$\BZ_2^{3\times 3}$:
\[
F_1=
\begin{pmatrix}
1&0&0\\
0&1&0\\
1&0&0
\end{pmatrix}\,,
\qquad
F_2=
\begin{pmatrix}
1&0&0\\
0&1&0\\
0&0&0
\end{pmatrix}\,,
\qquad
G=
\begin{pmatrix}
1&0&0\\
1&1&0\\
0&0&1
\end{pmatrix}\,.
\]
Define a binary operation $*$ on $A$ by $u*v=F_1u+F_2v$
(matrix-vector multiplication) and a unary operation
$g$ on $A$ by $g(v)=Gv+(1,0,0)^T$, so
\[
g:
\begin{pmatrix}
a\\
b\\
c
\end{pmatrix}
\mapsto
\begin{pmatrix}
a+1\\
a+b\\
c
\end{pmatrix}\,.
\]
These are affine operations, so the algebra $\m A=\langle
A;*,g\rangle$ is abelian, and $x-y+z$ is a compatible
operation. Define three subgroups of $\langle A;+\rangle$ as
follows: $B$ is the subgroup generated by $(0,0,1)^T$, $C$
is the subgroup generated by $(0,0,1)^T$ and $(0,1,0)^T$,
finally $D$ is the subgroup generated by $(0,1,0)^T$.  Let
$\beta$, $\gamma$ and $\delta$ denote the corresponding
congruences of $\langle A;+\rangle$. It is easy to check by
hand or by computer that $\m A$ has only these three
nontrivial proper congruences. We have $\beta\wedge\delta =0_A$ and
$\beta\vee\delta =\gamma$. It can be verified also that
$\Ho(\m A)$ is abelian.

The type of $\langle 0_A,\beta\rangle$ is $\utyp$, while
both $\langle\beta,\gamma\rangle$ and
$\langle\gamma,1_A\rangle$ have type~$\atyp$. These latter
quotients have the same minimal sets. There are four such
type~$\atyp$ minimal sets, one of which is
$U=\{(0,0,0)^T,(0,1,0)^T,(1,0,0)^T,(1,1,0)^T\}$. Thus the
induced algebra $\m A|_U$ is affine (since $\m A$ is
abelian). We also have that $U*U=A$, so the initial
condition in Theorem~\ref{spreadaff} is satisfied.

The algebra $\m A^2/\Delta$ is an $8$-element abelian
algebra (although it has a nonabelian quotient, so $\Ho(\m
A^2)$ is not abelian). Nevertheless, the abelianness of $\m
A^2/\Delta$ and $\Ho(\m A)$ are sufficient to make the proof
of statement~(1) of Theorem~\ref{spreadaff} work for~$\m A$
(which indeed has a compatible $+$ operation). But $\m A$ is
not affine, so we cannot drop the assumption from
statement~(2) in Theorem~\ref{spreadaff} that ${\mathcal
  V}(\m A)$ is abelian.
\end{exmp}

\begin{exmp}
  Now let $A$, $*$, $g$ be as in the preceding example, and
  let $h$ be the transposition of $A$ which switches
  $(1,0,0)^T$ and $(1,0,1)^T$ (and fixes all other elements
  of $A$).  Let $\overline{\m a}=\langle
  A;*,g,h\rangle$. This algebra has only four congruences,
  the nontrivial proper ones are $\beta$ and $\gamma$ above. The
  quotients $\langle\beta,\gamma\rangle$ and
  $\langle\gamma,1_A\rangle$ still have type~$\atyp$, and
  they have the same minimal sets, but now there are $16$ of
  them. One of these is $U$ above, so the initial condition
  in Theorem~\ref{spreadaff} is satisfied again. It is also
  true that $\Ho(\overline{\m a})$ is abelian.  However, $h$
  is not linear, and $\overline{\m a}^2/\Delta$ is a
  $5$-element nonabelian algebra. Therefore it is not
  sufficient to assume in (1) of Theorem~\ref{spreadaff}
  that $\Ho(\m A)$ is abelian.
\end{exmp}

\section{Semisimple algebras}\label{fact}

An algebra is called semisimple if it is isomorphic to a
subdirect product of simple algebras, or equivalently, if
its maximal congruences intersect to zero.

\begin{thm}\label{simplefact}
  Let $\m A$ be a finite, semisimple, solvable algebra. If
  the growth rate of~$\m A$ is linear, then $\m A$ has a
  Maltsev polynomial, and so it is affine.
\end{thm}

(This establishes for semisimple abelian algebras
the extra implication (iv)$\Rightarrow$(i)
among the conditions of Theorem~\ref{basic}.)

We shall need the following corollary of~Lemma~\ref{proj}.

\begin{cor}\label{tame}
  Let $\m A$ be a finite solvable algebra,
  $\eta_1,\ldots,\eta_n$ maximal congruences of~$\m A$ and
  $\eta=\bigwedge_{i=1}^n \eta_i$. Suppose that
  \begin{enumerate}
  \item For each $1\le i\le n$, the $\lb\eta_i,1\rb$-minimal
    sets are the same (that is, $M_{\m A}(\eta_i,1)=M_{\m
      A}(\eta_j,1)$ for every~$i$ and~$j$).
  \item $A/\eta$ has no nontrivial strongly abelian
    quotient algebras.
  \end{enumerate}
  Then $\lb\eta,1\rb$ is a tame quotient, and if\/
  $\theta\ge\eta$ is a coatom in $\Con(\m A)$, then $M_{\m
    A}(\theta,1)=M_{\m A}(\eta_1,1)$.
\end{cor}

\begin{proof}
  Clearly, for every upper cover $\delta$ of~$\eta$, the
  quotient $\lb\eta,\delta\rb$ is perspective to some
  $\lb\eta_i,1\rb$ (and so has type~$\atyp$). Condition
  $(2)$ implies that for every maximal
  congruence~$\theta\ge\eta$ the quotient $\lb\theta,1\rb$
  also has type~$\atyp$. Therefore it is projective to some
  $\lb\eta_i,1\rb$ by Lemma~\ref{proj}. Thus, if $U$~is an
  $\lb\eta_1,1\rb$-minimal set, then it is minimal with
  respect to each such quotient $\lb\eta,\delta\rb$ and
  $\lb\theta,1\rb$, and also $U\in M_{\m A}(\eta,1)$. This
  implies that restriction to~$U$ is a $0,1$-separating map,
  so by Definition~2.6 of~\cite{hobby-mckenzie}, the
  quotient $\lb\eta,1\rb$ is tame.
\end{proof}

\begin{proof}[Proof of Theorem~\ref{simplefact}]
  Our assumption that the growth rate of $\m a$ is linear
  implies, by Corollary~\ref{abelian_cor}~(1), that
  \begin{equation}\label{no_strab_q}
  \text{no power of $\m a^n$ of has a nontrivial strongly abelian 
    homomorphic image.}
  \end{equation}
  In particular, if
  $\alpha_1,\ldots,\alpha_\ell$ is the list of all maximal
  congruences of~$\m A$,
  then for each~$i$, the solvability of $\m A$ implies that  
  $\m a/\alpha_i$ is abelian, and (\ref{no_strab_q})
  implies that it is not strongly abelian.
  Therefore, $\langle \alpha_i,1\rangle$ has type $\atyp$
  for each~$i$.
  It follows also that $\m A$ is 
  abelian, since it is a subdirect product of the abelian
  algebras $\m a/\alpha_i$. 

  Call two maximal congruences
  $\alpha_i$ and~$\alpha_j$ of $\m a$ equivalent, if
  the $\lb\alpha_i,1\rb$-minimal sets and the
  $\lb\alpha_j,1\rb$-minimal sets are the same. This is an
  equivalence relation. For each equivalence class, consider
  the intersection $\beta_i$ of its members. Thus we get
  congruences $\beta_1,\ldots,\beta_m$ of $\m a$ such that their
  intersection is zero. In view of (\ref{no_strab_q}), 
  Corollary~\ref{tame} shows that
  $\lb\beta_i,1\rb$ is a tame quotient for every~$i$, and
  the set of all coatoms above~$\beta_i$ forms an
  equivalence class, for every~$i$. Denote $\m A/\beta_i$ by
  $\m B_i$. Then $\m A$ can be viewed as a subdirect
  subalgebra of $\m B_1\times\dots\times\m B_m$.

  Fix a minimal set $U_i$ corresponding to the coatoms whose
  intersection is $\beta_i$. This is a minimal set for the
  tame quotient $\lb\beta_i,1\rb$. Since~$\m A$ is abelian,
  its type~$\atyp$ minimal sets are affine and E-minimal
  (see the remark at the beginning of Section~\ref{abvar}).
  This means that if $\alpha_i$ and $\alpha_j$ are not
  equivalent, then $\alpha_i$ collapses every
  $\lb\alpha_j,1\rb$-minimal set (by which we mean that
  every $\lb\alpha_j,1\rb$-minimal set is contained in a
  single $\alpha_i$-class). Therefore if $i\ne j$, then
  $\beta_i$ collapses~$U_j$.  Hence $\beta_i$ restricts
  trivially to~$U_i$ (that is, $\beta_i|_{U_i}=0_{U_i}$),
  since $\beta_1\wedge\ldots\wedge\beta_m=0$. In fact, $U_i$
  is a $\lb\beta_i,1\rb$-trace, and therefore $U_i/\beta_i$
  is polynomially isomorphic to a vector space over a finite
  field.

  Since $\m A$ is abelian and satisfies (\ref{no_strab_q}),
  Theorem~\ref{spreadnilp} implies that $\m A$ is a spread
  of its type~$\atyp$ minimal sets.  Applying
  Lemma~\ref{proj} to $\m a$ and its maximal congruences
  $\alpha_1,\ldots,\alpha_\ell$ we see that every type
  $\atyp$ quotient of $\m a$ has the same minimal sets as
  $\langle\alpha_i,1\rangle$ for some $\alpha_i$.  In a
  spread, each set can be replaced with a polynomially
  isomorphic one, therefore we can use the representatives
  $U_1,\ldots,U_m$. Thus there exists a polynomial~$t$ such
  that
  \begin{equation}\label{fullA}
  t(U_1,\ldots,U_1,U_2,\ldots,U_2,\ldots,U_m,\ldots,U_m)=A\,.
  \end{equation}
  Write this polynomial as
  $t(\wec{x}^1,\wec{x}^2,\ldots,\wec{x}^m)$, where
  $\wec{x}^i$ is the string of variables where $U_i$ occurs.

  Let $0_i\in U_i$ be fixed arbitrarily, and substitute
  $0_i$ to every variable in $\wec{x}^i$ for~$i\ne 1$. We
  get a polynomial
  \[
  f(\wec{x}^1)=f(x_1^1,\ldots,x_n^1)=
  t(\wec{x}^1,\hat 0_2,\ldots,\hat 0_m)
  \]
  of~$\m A$. Let $T_1=f(U_1,\ldots,U_1)$.

  Each $\beta_j$ collapses~$T_1$ whenever $j\ne 1$, and
  therefore $\beta_1$ restricts trivially to~$T_1$ (using
  again that
  $\beta_1\wedge\ldots\wedge\beta_m=0$). Therefore, as $U_1$
  is a $\lb\beta_1,1\rb$-trace, for every $a\ne b\in T_1$
  there exists a unary polynomial mapping $T_1$ to~$U_1$ that
  separates $a$ and~$b$. Thus we can apply Lemma~3.8
  of~\cite{kkv1} to $S=U_1$, zero element $0_1$ and~$f$. The
  proof of this lemma shows that there exist 
  \begin{enumerate}
  \item unary polynomials $g_1,\ldots,g_k$ mapping $T_1$
    to~$U_1$, that satisfy $g_i(0_1)=0_1$, and
  \item $k$-ary polynomials $\ell_1,\ldots,\ell_n$
    satisfying $\ell_i(\hat 0_1)=0_1$ and
    $\ell_i(U_1,\ldots,U_1)\subseteq U_1$
  \end{enumerate}
  such that for $f'(\wec{y})=
  f\big(\ell_1(\wec{y}),\ldots,\ell_n(\wec{y})\big)$ we have
  \[
  g_i\big(f'(y_1,\ldots,y_k)\big) = y_i
  \]
  when each $y_1,\ldots,y_k\in U_1$ and
  \begin{equation}\label{coord}
    f'\big(g_1(x),\ldots,g_k(x)\big) = x
  \end{equation}
  for every~$x\in T_1$. Thus $f'(U_1,\ldots,U_1)=T_1$. Note that
  $f'(\hat 0_1)=f(\hat 0_1)$ holds.

  We defined polynomials $g_i=g_i^1$ and $\ell_i=\ell_i^1$
  above for the first block of variables $\wec{x}^1$ of the
  polynomial $t(\wec{x}^1,\ldots,\wec{x}^m)$. Do this in an
  analogous way for all other blocks of variables
  $\wec{x}^j$ to obtain unary polynomials~$g_i^j$ ($1\le
  i\le k_j$) and $k_j$-ary polynomials~$\ell_i^j$ ($1\le
  i\le n_j$), and the set~$T_j$. We shall substitute these
  polynomials into $t$ in the following way. Consider the
  $j$-th block $\wec{x}^j=(x_1^j,\ldots,x_{n_j}^j)$ of the
  variables of $t$. Then make the substitution
  \[
  x_i^j \to \ell_i^j\big(g_1^j(y_j),\ldots,g_{k_j}^j(y_j)\big)
  \]
  for $1\le i\le n_j$ and $1\le j\le m$. This way, we get a
  polynomial $r(y_1,\ldots,y_m)$.

  Equation~(\ref{coord}) and its analogues ensure that
  \[
  r(0_1,\ldots,0_{j-1},c,0_{j+1},\ldots,0_m)=c
  \]
  holds whenever $c\in T_j$. Thus if $c_j\in T_j$, then
  $r(c_1,\ldots,c_m)\equiv_{\beta_j}c_j$. 
  
  Let $a\in A$. By equation~(\ref{fullA}) we have $a =
  t(\wec{a}^1,\ldots,\wec{a}^m)$ for appropriate vectors
  $\wec{a}^j$ such that each component of $\wec{a}^j$ is
  in~$U_j$. Replace each $\wec{a}^j$ by $\hat 0_j$ for $j\ge
  2$, and call the resulting element~$c_1\in T_1$. Since
  each $U_j$ is contained in a $\beta_i$-class for
  every~$i\ne j$ we see that $c_1\equiv_{\beta_1} a$. Do the
  analogous substitutions for all other variables to obtain
  elements $c_i\in T_i$ such that $a\equiv_{\beta_i}
  c_i$. Then $r(c_1,\ldots,c_m)\equiv_{\beta_i}a$ for
  every~$i$, and since the $\beta_i$-s intersect to zero we
  have that $r(c_1,\ldots,c_m)=a$. 

  We show that if $c_j'\in T_j$ are such that
  $r(c_1',\ldots,c_m')=a$, then $c_j=c_j'$ for
  every~$j$. Indeed, $c_j \equiv_{\beta_j}
  a=r(c_1',\ldots,c_m')\equiv_{\beta_j}c_j'$, and $\beta_j$
  restricts trivially to~$T_j$.

  In other words, $r:T_1\times\dots\times T_m\to A$ is a
  bijection. Corollary~3.7 of~\cite{kkv1} states that $T_j$
  is the range of an idempotent polynomial~$e_j$, and the
  induced algebra on $T_j$ is isomorphic to an algebra
  polynomially equivalent to a full matrix power of $\m
  A|_{U_j}$. Therefore $T_j$ has an induced Maltsev
  polynomial $d_j$. Let
  \[
  d(x,y,z)=r(\ldots,d_j\big(e_j(x),e_j(y),e_j(z)\big),\ldots)\,.
  \]
  We prove that $d(x,x,z)=r(\ldots,e_j(z),\ldots)=z$.  Let
  $z=r(z_1,\ldots,z_m)$, where $z_j\in T_j$. Then
  $z\equiv_{\beta_j} z_j$, so $e_j(z)\equiv_{\beta_j}
  e_j(z_j)=z_j$, and so $e_j(z)=z_j$, since $\beta_j$
  restricts trivially to~$T_j$. Similarly,
  $d(x,z,z)=x$. Hence $d$ is a Maltsev polynomial of~$\m A$
  and the proof of~ Theorem~\ref{simplefact} is complete.
\end{proof}

We actually proved that $\m A$ is the full direct product of
the algebras~$\m B_i$, since $T_i$ and $B_i$ are in a
bijective correspondence via factoring modulo~$\beta_i$.

The following example shows that the direct product of two
affine algebras need not have a Maltsev polynomial, even if its growth rate
is linear, and so Theorem~\ref{simplefact} cannot be
improved to say that if $\m A/\alpha_1$ and $\m A/\alpha_2$
have Maltsev polynomials, then so does $\m A/(\alpha_1\wedge \alpha_2)$.

\begin{exmp}\label{abelian_spread_not_maltsev}
Consider $V=\BZ_2^2$ as a vector space over the two-element
field $\BZ_2^2$. Let
\[
P=
\begin{pmatrix}
1&0\\
0&0
\end{pmatrix}
\qquad\text{and}\qquad
N=
\begin{pmatrix}
0&0\\
1&0
\end{pmatrix}\,.
\]
Define an operation $f$ on $V$ by
\[
f(\wec{x},\wec{y})= P\wec{x}+N\wec{y}.
\]
Clearly, $P^2=P$, $PN=N^2=0$ and $NP=N$. Therefore
$\{0,P,N\}$ is a semigroup under matrix multiplication, and
the clone of all term operations of the algebra 
$\langle V;+,f,\wec{0}\rangle$ consists
of all functions of the form
\[
M_1\wec{x}_1+\ldots+M_n\wec{x}_n\,,
\]
where each $M_i$ is either $P$, $N$, or the identity
matrix. If one deletes $+$ from the basic operations, then
the identity matrix cannot occur as a coefficient (except
for projections), and there can be at most one instance of
$P$ and at most one instance of~$N$. It is left to the
reader to check that $f(V,V)=V$.

Next we define two algebras $\m B$ and $\m C$. The language
contains the operation symbols $+$, $\oplus$, a $4$-ary $g$,
and $0$. In the algebra~$\m B$ the underlying set is~$V$,
$+$ interprets as the usual addition, $\oplus$ as the
constant zero function, and $0$ as the zero
vector~$\wec{0}$. In the algebra $\m C$, the only difference
is that $+$ is interpreted as constant zero, and $\oplus$ as
the usual addition of~$V$. Finally,
\begin{align*}
  g^{\m B}(\wec{x},\wec{y},\wec{u},\wec{v}) &=
  P\wec{x}+N\wec{y}=f(\wec{x},\wec{y}),\\
  g^{\m C}(\wec{x},\wec{y},\wec{u},\wec{v}) &=
  P\wec{u}\oplus N\wec{v}=f(\wec{u},\wec{v})\,.
\end{align*}

Both algebras $\m b$ and $\m c$ have Maltsev terms, namely
$x-y+z$ for $\m b$ and $x\ominus y\oplus z$ for $\m c$. 
In addition, both algebras are abelian, hence affine.
Now consider the algebra 
$\m A=\m B\times \m C$. It follows that $\m a$ is also abelian.
Our goal is to show that $\m a$ has linear growth rate, 
but has no Maltsev polynomial. 

To see that the growth rate of $\m a$ is linear, observe
first that $e_B(x)=x+(\wec{0},\wec{0})$ is an idempotent
unary polynomial of $\m A$ mapping $A=B\times C$ to the
subset $U=B\times \{\wec{0}\}$. Similarly, $e_C(x)=x\oplus
(\wec{0},\wec{0})$ is an idempotent unary polynomial of $\m
A$ mapping $A$ to the subset $W=\{\wec{0}\}\times C$. The
induced algebras $\m A|_U$ and $\m A|_W$ are abelian and
have Maltsev polynomials, hence they are affine algebras.
Moreover,
\[
g^{\m A}(U,U,W,W)=A.
\]
Thus $\m a$ has linear growth rate by Theorem~\ref{affine_prep},
Lemma~\ref{spreadaffrate}
and Corollary~\ref{abelian_cor}~(2).

It remains to prove that $\m A$ does not have a Maltsev polynomial. 
Suppose that $\m a$ has a Maltsev polynomial.
Since $\m a$ is abelian, we get that $\m a$ is affine, and hence it
has a Maltsev term $M(x,y,z)$.
Since $\m b$, $\m c$ are affine, each one has a unique Maltsev term, so
$M^{\m b}(x,y,z)=x-y+z$ and $M^{\m c}(x,y,z)=x\ominus y\oplus z$.
Thus $M^{\m a}$ acts componentwise as $x-y+z$ in the $\m b$-coordinate 
and as $x\ominus y\oplus z$ in the $\m c$-coordinate.
In particular, if $\wec{v}=(0,1)^T\in V$, then
for the elements
$(\wec{v},\wec{0}),(\wec{0},\wec{0}),(\wec{0},\wec{v})$ of $A$ we get that
$M^{\m a}\bigl((\wec{v},\wec{0}),(\wec{0},\wec{0}),(\wec{0},\wec{v})\bigr)
=(\wec{v},\wec{v})$.
This implies that
$\{(\wec{v},\wec{0}),(\wec{0},\wec{0}),(\wec{0},\wec{v})\}$
is not a subuniverse of $\m a$.
On the other hand, it is not hard to check that the set 
$\{(\wec{v},\wec{0}),(\wec{0},\wec{0}),(\wec{0},\wec{v})\}$
is closed under all operations $+$, $\oplus$, $g$, $0$ of $\m a$.
Indeed, closure under $+$ and $\oplus$ follows, 
because $+$ and $\oplus$ are constant $\wec{0}$ 
in one of the components and the usual addition of $V$ in the other,
and $\{\wec{v},\wec{0}\}$ is closed under addition.
Finally, closure under $g$ follows from the fact that
we have
$f(\wec{v},\wec{v})=f(\wec{v},\wec{0})=
f(\wec{0},\wec{v})=\wec{0}$, since
$P\wec{v}=N\wec{v}=\wec{0}$. 
This shows that
$\{(\wec{v},\wec{0}),(\wec{0},\wec{0}),(\wec{0},\wec{v})\}$
is a subuniverse of $\m a$. The contradiction obtained proves that
$\m a$ has no Maltsev polynomial.
\end{exmp}

In this example, $\m A$ is abelian, $d_{\m A}(n)\in O(n)$,
and yet $\m A$ is not affine. This proves that,
for abelian algebras, 
no one of the equivalent conditions (iii), (iv), (v)
or (vi) implies any one of the equivalent conditions
(i) or (ii).

\section{Summary of results. Problems.}\label{summary_sec}

We are now in a position to prove Theorem~\ref{basic}, which
summarizes our main results.

\begin{proof}[Proof of Theorem~\ref{basic}]
[(i)$\Rightarrow$(iv)]
A Maltsev polynomial is a $3$-ary,
$0$-pointed, $2$-cube term. Hence this implication
follows from Corollary~\ref{pointed_polynomial_cor}
which implies that, if $\m a$ is an
algebra with a $0$-pointed, $k$-cube polynomial,
and $\m a^k$ is finitely generated, 
then $d_{\m a}(n)\in O(n^{k-1})$.

[(i)$\Rightarrow$(ii)]
The definition of a pointed cube 
polynomial generalizes that of a Maltsev polynomial.

[(ii)$\Rightarrow$(v)]
Theorem~\ref{pointed_polynomial} shows that
if $\m a$ is an
algebra with a $p$-pointed $k$-cube polynomial,
and $p\geq 1$, then $d_{\m a}(n)$
is bounded above by a polynomial in $n$
if $\m a^{p+k-1}$ is finitely generated.
The restriction on $\m a^{p+k-1}$ is 
satisfied when $\m a$ is finite.
This proves that (ii) implies (v) when $p\geq 1$.
The case $p=0$ is handled similarly using
Corollary~\ref{pointed_polynomial_cor}.

[(iii)$\Rightarrow$(iv)]
A binary polynomial with a unit element is 
a $2$-ary, $1$-pointed, $2$-cube polynomial.
Corollary~\ref{pointed_polynomial_cor}
implies that the growth rate of any finite
algebra that has such a polynomial 
lies in $O(n)$. The induced algebra on
an $\langle\alpha,\beta\rangle$-minimal set
of type $\atyp, \btyp, \ltyp$
or $\styp$ has such a 
binary polynomial: take $d(x,0,y)$ with $d$ a pseudo-Maltsev
polynomial and $0$ in the body
if the type is $\atyp$, and take $x\wedge y$
with $\wedge$ a pseudo-meet polynomial in the other cases.
Therefore, by Lemma~\ref{spreadaffrate} of this paper, 
$d_{\m a}(n)\in O(n)$ whenever $\m a$ is a spread
of minimal sets whose types are not $\utyp$.

[(iv)$\Rightarrow$(v)]
$O(n)\cap 2^{\Omega(n)}=\emptyset$.

[(v)$\Rightarrow$(vi)] Theorem~\ref{first_bounds} 
implies that $d_{\m a}(n)\notin 2^{\Theta(n)}$ is equivalent to
$d_{\m a}(n)\notin 2^{\Omega(n)}$ for finite algebras.
Hence (v)$\Rightarrow$(vi) is just the contrapositive of 
Corollary~\ref{abelian_cor}~(1).

[(vi)$\not\Rightarrow$(v)]
Example~5.3.5 of \cite{kksz-A}
describes finite implication 
algebras with exponential growth. These satisfy
(vi), since no nontrivial implication algebra is strongly
abelian, but do not satisfy (v).

[(ii)$\not\Rightarrow$(iv), (v)$\not\Rightarrow$(iv)]
Given $k\geq 2$, Theorem~5.3.1 of \cite{kksz-A} constructs
a finite algebra with a cube polynomial whose $d$-function
satisfies $d_{\m a}(n)\in \Theta(n^{k-1})$.
When $k=3$ one has 
$d_{\m a}(n)\notin 2^{\Omega(n)}$, yet
$d_{\m a}(n)\notin O(n)$.

[(i)$\not\Rightarrow$(iii), (iv)$\not\Rightarrow$(iii)]
If $\m a$ is a 2-element Boolean algebra, then
$d_{\m a}(n)\in O(\log(n))$ (hence $d_{\m a}(n)\in O(n)$), 
and therefore (i) and (iv) hold. 
But (iii) does not hold, since $\m a$ has no 
type $\atyp$ minimal sets.

[(iii)$\not\Rightarrow$(ii)]
Example~\ref{abelian_spread_not_maltsev}
describes 
an abelian
algebra $\m c$ that is a spread of type 
$\atyp$ minimal sets, but does not have
a Maltsev term. If $\m c$ had a pointed cube
polynomial, then by Theorem~\ref{nilpotent}
it would have Maltsev polynomial.
But it is well known that an abelian algebra
with a Maltsev polynomial has a Maltsev
term, and $\m c$ does not have such a term.

[(iv)$\not\Rightarrow$(ii)]
According to Theorem~\ref{avoid},
any function that can be 
realized as the growth rate of a finite algebra
can also be realized as the growth rate of a finite
algebra that does not have a pointed cube polynomial.
\end{proof}

Recall the six
growth-restricting conditions portrayed in this theorem:
\begin{enumerate}
\item[(i)]
$\m a$ has a Maltsev polynomial.
\item[(ii)]
$\m a$ has a pointed cube polynomial.
\item[(iii)] $\m A$ is a spread of its type~$\atyp$ minimal sets.
\item[(iv)]
$d_{\m a}(n)\in O(n)$.
\item[(v)]
$d_{\m a}(n)\notin 2^{\Omega(n)}$.
\item[(vi)]
No finite power $\m a^n$ has a nontrivial
strongly abelian homomorphic image.
\end{enumerate}

We have shown that for arbitrary finite algebras, the
following implications hold:
\begin{center}
\begin{picture}(200,70)
\setlength{\unitlength}{1mm}

\put(0,0){$(iii)$}
\put(10,0){$\Longrightarrow$}
\put(20,0){$(iv)$}
\put(30,0){$\Longrightarrow$}
\put(40.5,0){$(v)$}
\put(50,0){$\Longrightarrow$}
\put(60,0){$(vi)$.}
\put(22,10){\rotatebox[origin=c]{270}{$\Longrightarrow$}}
\put(21,20){$(i)$}
\put(30,20){$\Longrightarrow$}
\put(40,20){$(ii)$}
\put(42,10){\rotatebox[origin=c]{270}{$\Longrightarrow$}}
\end{picture}
\end{center}

If $\m a$ is a finite solvable algebra, then this can be
strengthened to
\begin{center}
\begin{picture}(200,70)
\setlength{\unitlength}{1mm}

\put(0,0){$(iii)$}
\put(10,0){$\Longrightarrow$}
\put(20,0){$(iv)$}
\put(30,0){$\Longrightarrow$}
\put(40.5,0){$(v)$}
\put(50,0){$\Longrightarrow$}
\put(60,0){$(vi)$.}
\put(10,10){\rotatebox[origin=c]{225}{$\Longrightarrow$}}
\put(21,20){$(i)$}
\put(30,20){$\Longrightarrow$}
\put(40,20){$(ii)$}
\put(30,10){\rotatebox[origin=c]{225}{$\Longrightarrow$}}
\end{picture}
\end{center}

If $\m a$ is a finite left nilpotent algebra, then 
we have established that 
\begin{center}
\begin{picture}(200,70)
\setlength{\unitlength}{1mm}
\put(0,0){$(iii)$}
\put(10,0){$\Longleftrightarrow$}
\put(20,0){$(iv)$}
\put(30,0){$\Longleftrightarrow$}
\put(40.5,0){$(v)$}
\put(50,0){$\Longleftrightarrow$}
\put(60,0){$(vi)$.}
\put(10,10){\rotatebox[origin=c]{225}{$\Longrightarrow$}}
\put(21,20){$(i)$}
\put(30,20){$\Longleftrightarrow$}
\put(40,20){$(ii)$}
\put(30,10){\rotatebox[origin=c]{225}{$\Longrightarrow$}}
\end{picture}
\end{center}

Finally, if $\m a$ is a semisimple abelian algebra or if it
generates an abelian variety, then all six conditions are
equivalent. On the other hand, we gave an example of an
abelian algebra satisfying the properties in the bottom row
but not satisfying those in the top row, so no other
implications hold for finite abelian or nilpotent algebras.

Now let us return to the ``solvability'' diagram. The
example preceding Theorem~6.10 in \cite{kkv1}, which is a
finite, solvable algebra which has no Maltsev polynomial,
but has a binary polynomial with a unit element, shows that
(ii)$\not\Rightarrow$(i) for finite solvable algebras. We
have seen that no item on the bottom row implies any item on
the top row for solvable algebras, so the true implications
yet to be discovered can only be (ii)$\Rightarrow$(iii), or
the reversal of some of the implications along the bottom
row. This suggests some problems.

\begin{prb}
  Does (ii)$\Rightarrow$(iii) hold for finite solvable
  algebras?
\end{prb}

\begin{prb}
Which of the true implications 
(iii)$\Rightarrow$(iv)$\Rightarrow$(v)$\Rightarrow$(vi)
can be reversed for finite solvable algebras?
\end{prb}

\bibliographystyle{plain}

\end{document}